\documentclass[12pt,reqno]{amsart}

\usepackage[usenames]{color}
\usepackage{amssymb}
\usepackage{graphicx}
\usepackage{amscd}
\usepackage{url}
\usepackage[colorlinks=true,
linkcolor=webgreen,
filecolor=webbrown,
citecolor=webgreen]{hyperref}

\definecolor{webgreen}{rgb}{0,.5,0}
\definecolor{webbrown}{rgb}{.6,0,0}

\usepackage{color}
\usepackage{fullpage}
\usepackage{float}
\usepackage{comment}
\usepackage{cite}
\usepackage{url}
\usepackage{graphics,amsmath,amssymb}
\usepackage{amsthm}
\usepackage{amsfonts}
\usepackage{latexsym}
\usepackage{enumerate}

\theoremstyle{plain}\numberwithin{equation}{section}
\newtheorem{theorem}{Theorem}[section]
\newtheorem{corollary}[theorem]{Corollary}
\newtheorem{lemma}[theorem]{Lemma}
\newtheorem{proposition}[theorem]{Proposition}

\theoremstyle{remark}

\allowdisplaybreaks
\newcommand{\Ft}[1]{\mathcal{F}_{#1}}
\newcommand{\Lt}[1]{\mathcal{L}_{#1}}

\newcommand{\lc}{{\rm lc}}

\begin{document}
	
\title[Irreducibility of generalized Fibonacci polynomials]
       {Irreducibility of generalized Fibonacci polynomials}

\author[Rigoberto Fl\'orez]{Rigoberto Fl\'orez }
\address{Department of Mathematical Sciences\\ The Citadel\\ Charleston, SC \\U.S.A}
\email{rigo.florez@citadel.edu}
\thanks{}

\author[J.~C.~Saunders]{J.~C.~Saunders}
\address{Department of Mathematics and Statistics\\ University of Calgary\\AB T2N 1N4\\Canada}
\email{john.saunders1@ucalgary.ca}
\thanks{}

\begin{abstract}
A second order polynomial sequence is of \emph{Fibonacci-type $\Ft{n}$} (\emph{Lucas-type} $\Lt{n}$) if its Binet formula has a structure similar to 
that for Fibonacci (Lucas) numbers. Under certain conditions these polynomials are irreducible if and only if $n$ is a prime number. For example, the 
Fibonacci polynomials, Pell polynomials, Fermat polynomials, Lucas polynomials, Pell-Lucas polynomials, Fermat-Lucas polynomials are irreducible  
when $n$ is a prime number; and  Chebyshev polynomials (second kind), Morgan-Voyce polynomials (Fibonacci type), and Vieta polynomials 
are reducible when $n$ is a prime number. 

In this paper we give some theorems to determine whether the Fibonacci type polynomials and Lucas type polynomials are irreducible when $n$ is prime. 

\end{abstract}

\keywords{Irreducible polynomial, Prime number, Extension field, Fibonacci polynomial, Lucas polynomial.}

\maketitle

\section{Introduction}
The Fibonacci polynomials $F_{n}$ are defined as $F_{n}(x) = x F_{n - 1}(x) + F_{n - 2}(x)$ where $F_{0}(x)=0$ and $F_{1}(x)=1$. Webb et al. 
\cite{WebbParberry} proved that  $F_{p}$ is irreducible if and only if $p$ is a prime number.  Hogatt et al. \cite{HoggattLong} defined a bivariate 
generalized Fibonacci polynomial  $u_n(x,y)$ and proved that $u_p(x,y)$ is irreducible over $\mathbb{Q}$ if and only if $p$ is a prime number.

The Lucas polynomials $L_{n}$ are defined as $L_{n}(x) = x L_{n - 1}(x) + L_{n - 2}(x)$ where $L_{0}(x)=2$ and $L_{1}(x)=x$.
Bergum and Hoggatt \cite{HoggattBergum} proved that  $L_{p}(x)/L_{1}(x)$ is irreducible if and only if $p>2$ is a prime number. They also defined a bivariate  
generalized Lucas polynomial  $v_n(x,y)$ and proved that $v_p(x,y)/v_1(x,y)$ is irreducible over $\mathbb{Q}$ if and only if $p>2$  is a prime number.

A second order polynomial sequence is of \emph{Fibonacci-type} (\emph{Lucas-type}) if its Binet formula
has a structure similar to that for Fibonacci (Lucas) numbers. Those are known as \emph{generalized Fibonacci polynomials}  GFP
(see  \cite{florezHiguitaAnotherProof, florezHiguitaRamirez,florezHiguitaMuk2018, florezHiguitaMuk:StarD,FlorezMcAnallyMuk}). Some known examples are: 
Pell polynomials, Fermat polynomials, Chebyshev polynomials, Morgan-Voyce polynomials, Lucas polynomials, Pell-Lucas polynomials,  
Fermat-Lucas polynomials, Chebyshev polynomials, Vieta polynomials and Vieta-Lucas polynomials. Other generalized Fibonacci polynomials are in 
  \cite{Richard,HoggattBergum, HoggattLong}.

From the discussion in the first two paragraphs above, we have two natural questions, is it true that $\Ft{p}(x)$ is irreducible if and only if $p$  
is a prime number?   And, is it true that $\Lt{p}(x)/\Lt{1}(x)$ is irreducible if and only if $p>2$ is a prime number? In this paper we give precise conditions  
to determine whether some families of GFP are irreducible when $p$ is a prime number and give precise conditions to determine whether some families 
of GFP are reducible when $p$ is a prime number.  As a corollary of the  theorems proved here, we obtain that the Fibonacci polynomials, the  
Pell polynomials, the Fermat polynomials, are irreducible  when $p>0$ is a prime number. A second corollary is that Chebyshev polynomials  
(second kind), Morgan-Voyce polynomials (Fibonacci type), and Vieta polynomials are reducible when $p$ is a  prime number. As a third  corollary we have that  
$\Lt{p}(x)/\Lt{1}(x)$ is irreducible where $p>2$ is a prime number  and  
$\Lt{p}(x)$ is one of these: Lucas polynomials, Pell-Lucas polynomials, Fermat-Lucas polynomials.    
		
\section{Second order polynomial sequences} 
In this section we reproduce the definitions by Fl\'orez et. al. 
\cite{florezHiguitaAnotherProof, florezHiguitaRamirez,florezHiguitaMuk2018, florezHiguitaMuk:StarD,FlorezMcAnallyMuk}  
for generalized Fibonacci polynomials. The definitions here give rise to some known polynomial sequences (see for example, Table \ref{familiarfibonacci} or  
\cite{florezHiguitaAnotherProof, florezHiguitaRamirez,florezHiguitaMuk2018, florezHiguitaMuk:StarD,FlorezMcAnallyMuk,HoggattLong, Pell, Fermat,  koshy}).
Throughout the paper we consider polynomials in $\mathbb{Q}[x]$  or in $\mathbb{Z}[x]$. 

We now give the two second order polynomial recurrence relations in which we divide the generalized Fibonacci polynomials (GFP):
\begin{equation}\label{Fibonacci;general:FT}
\Ft{0}(x)=0, \; \Ft{1}(x)= 1,\;  \text{and} \;  \Ft{n}(x)= d(x) \Ft{n - 1}(x) + g(x) \Ft{n - 2}(x) \text{ for } n\ge 2,
\end{equation}
where $d(x)$, and $g(x)$ are fixed non-zero polynomials in $\mathbb{Z}[x]$ with $\gcd(d(x), g(x))=1$.

We say that a polynomial recurrence relation is of \emph{Fibonacci-type} if it satisfies the relation given in
\eqref{Fibonacci;general:FT}, and of \emph{ Lucas-type} if:
\begin{equation}\label{Fibonacci;general:LT}
\Lt{0}(x)=p_{0}, \; \Lt{1}(x)= p_{1}(x),\;  \text{and} \;  \Lt{n}(x)= d(x) \Lt{n - 1}(x) + g(x) \Lt{n - 2}(x) \text{ for } n\ge 2,
\end{equation}
where $|p_{0}|=1$ or $2$ and $p_{1}(x)$, $d(x)=\alpha p_{1}(x)$, and $g(x)$ are fixed non-zero  polynomials in $\mathbb{Z}[x]$ with
$\alpha$ an integer of the form $2/p_{0}$.  Some known examples
of Fibonacci-type polynomials and  Lucas-type polynomials are in Table \ref{familiarfibonacci} or in
\cite{florezHiguitaAnotherProof, florezHiguitaRamirez,florezHiguitaMuk2018, florezHiguitaMuk:StarD,FlorezMcAnallyMuk,HoggattLong,Pell, Fermat,  koshy}.

If $G_{n}$ is either $\Ft{n}$ or $\Lt{n}$ for all $n\ge 0$ and $d^2(x)+4g(x)> 0$, then the explicit formula for the recurrence relations in
 \eqref{Fibonacci;general:FT} and \eqref{Fibonacci;general:LT}  is given by
\begin{equation*}
 G_{n}(x) = t_1 a^{n}(x) + t_2 b^{n}(x)
\end{equation*}
where $a(x)$ and $b(x)$ are the solutions of the quadratic equation associated with the second order
recurrence relation $G_{n}(x)$. That is,  $a(x)$ and $b(x)$ are the solutions of $z^2-d(x)z-g(x)=0$.
If $\alpha=2/p_{0}$, then the Binet formula for Fibonacci-type polynomials is stated in  \eqref{bineformulauno} and the Binet formula
for  Lucas-type polynomials is stated in \eqref{bineformulados}
(for details on the construction of the two Binet formulas see \cite{florezHiguitaMuk2018})
\begin{equation}\label{bineformulauno}
\Ft{n}(x) =\dfrac{a^{n}(x)-b^{n}(x)}{a(x)-b(x)}
\end{equation}
and
\begin{equation}\label{bineformulados}
\Lt{n}(x)=\dfrac{a^{n}(x)+b^{n}(x)}{\alpha}.
\end{equation}
Since $a(x)$ and $b(x)$ are solutions of  $z^2-d(x)z-g(x)=0$, we have 
$$a(x)+b(x)=d(x), \quad a(x)b(x)= -g(x), \quad \text{ and } \quad a(x)-b(x)=\sqrt{d^2(x)+4g(x)}$$
where $d(x)$ and $g(x)$ are the polynomials defined in \eqref{Fibonacci;general:FT} and \eqref{Fibonacci;general:LT}.
These give that 

\begin{equation}\label{eqn1}
a(x)=\frac{d(x)+\sqrt{d^2(x)+4g(x)}}{2} \quad \text{ and } b(x)=\frac{d(x)-\sqrt{d^2(x)+4g(x)}}{2}.
\end{equation}

A sequence of  Lucas-type (Fibonacci-type) is \emph{equivalent} or \emph{conjugate} to a sequence of Fibonacci-type (Lucas-type),
if their recursive sequences are determined by the same polynomials $d(x)$ and $g(x)$. Notice that two equivalent polynomials
have the same $a(x)$ and $b(x)$ in their Binet representations. In \cite{florezHiguitaMuk2018, florezHiguitaMuk:StarD,koshy,koshyPoly}
there are examples of some known equivalent polynomials with their Binet formulas.  The  polynomials in Tables \ref{familiarfibonacci}  and \ref{Equivalent} are
organized by pairs of equivalent polynomials.  For instance,  Fibonacci and Lucas, Pell and Pell-Lucas, and so on.

We use $\deg(P)$ and $\lc(P)$ to mean the degree and the leading  coefficient of a polynomial  $P$.
Most of the following conditions were required in the papers that we are citing. Therefore, we require here that
$\gcd(d(x), g(x))=1$ and $\deg(g(x))< \deg(d(x))$ for both types of sequences
 and  that conditions in \eqref{extra:condition} also hold for Lucas types polynomials;
 \begin{equation}\label{extra:condition}
 \gcd(p_{0}, p_{1}(x))=1,  \gcd(p_{0}, d(x))=1, \gcd(p_{0}, g(x))=1,  \text{ and  that degree of } \Lt{1} \ge1.
 \end{equation}

Notice that in the definition of Pell-Lucas we have
$Q_{0}(x)=2$ and $Q_{1}(x)=2x$. Thus, the $\gcd(2,2x)=2\ne 1$.
Therefore, Pell-Lucas does not satisfy the extra conditions that we imposed in \eqref{extra:condition}. So,
to resolve this inconsistency we use $Q^{\prime}_{n}(x)=Q_{n}(x)/2$ instead of $Q_{n}(x)$.
	
\begin{table}[!ht]
		\begin{center}\scalebox{0.8}{
				\begin{tabular}{|l|l|l|l|} \hline
					Polynomial                      & Initial value     & Initial value	& Recursive Formula 						       \\	
					&$G_0(x)=p_0(x)$  	      &$G_1(x)=p_1(x)$	&$G_{n}(x)= d(x) G_{n - 1}(x) + g(x) G_{n - 2}(x)$ 	   \\  \hline   \hline
					Fibonacci             	      & $0$	     &$1$      	&$F_{n}(x) = x F_{n - 1}(x) + F_{n - 2}(x)$	 	       \\
					Lucas 	             	      &$2$	     & $x$ 	 	&$D_n(x)= x D_{n - 1}(x) + D_{n - 2}(x)$                \\ 						
					Pell			    	      &$0$	     & $1$           &$P_n(x)= 2x P_{n - 1}(x) + P_{n - 2}(x)$               \\
					Pell-Lucas 	    	      &$2$	     &$2x$          &$Q_n(x)= 2x Q_{n - 1}(x) + Q_{n - 2}(x)$               \\
					Pell-Lucas-prime 	      &$1$	     &$x$       	&$Q_n^{\prime}(x)= 2x Q_{n - 1}^{\prime}(x) + Q_{n - 2}^{\prime}(x)$               \\
					Fermat  	                       &$0$	     & $1$      	&$\Phi_n(x)= 3x\Phi_{n-1}(x)-2\Phi_{n-2}(x) $           \\
					Fermat-Lucas	              &$2$	     &$3x$  	&$\vartheta_n(x)=3x\vartheta_{n-1}(x)-2\vartheta_{n-2}(x)$\\
					Chebyshev second kind &$0$	     &$1$       	&$U_n(x)= 2x U_{n-1}(x)-U_{n-2}(x)$  	 	              \\
					Chebyshev first kind       &$1$	     &$x$       	&$T_n(x)= 2x T_{n-1}(x)-T_{n-2}(x)$  \\	 	
					Morgan-Voyce	              &$0$	     &$1$      	&$B_n(x)= (x+2) B_{n-1}(x)-B_{n-2}(x) $  	 	         \\
					Morgan-Voyce 	              &$2$	     &$x+2$       &$C_n(x)= (x+2) C_{n-1}(x)-C_{n-2}(x)$  	 	         \\
					Vieta 		              &$0$ 	     &$1$	        &$V_n(x)=x V_{n-1}(x)-V_{n-2}(x)$ 	    \\
					Vieta-Lucas 		      &$2$ 	     &$x$	        &$v_n(x)=x v_{n-1}(x)-v_{n-2}(x)$      \\
					\hline
			\end{tabular}}
		\end{center}
		\caption{Recurrence relation of some GFP.} \label{familiarfibonacci}
	\end{table}

\begin{table} [h]
\begin{center}\scalebox{0.8}{
\begin{tabular}{|lc|lc|l|l|} \hline
   Polynomial  	    & $\Lt{n}(x)$    	     &Polynomial of 	&$\Ft{n}(x)$   &$a(x)$ 	              	& $b(x)$			\\	
   Lucas type 	 	& 			     & Fibonacci type  	&			& 					        &      				\\ \hline \hline
   Lucas 			&$D_n(x)$     	     &Fibonacci 		&$F_n(x)$ 	&  $(x+\sqrt{x^2+4})/2$     & $(x-\sqrt{x^2+4})/2$ \\ 						
   Pell-Lucas 		&$Q_n(x)$	 	     &Pell				& $P_n(x)$  &  $x+\sqrt{x^2+1}$	      	& $x-\sqrt{x^2+1}$        \\
   Fermat-Lucas 	& $\vartheta_n(x)$	 & Fermat 			& $\Phi_n(x)$ &  $(3x+\sqrt{9x^2-8})/2$ & $(3x-\sqrt{9x^2-8})/ 2$ \\
   Chebyshev first kind& $T_n(x)$ 		 &Chebyshev second kind&$U_n(x)$ &  $x +\sqrt{x^2-1}$      	& $x -\sqrt{x^2-1}$       \\
   Morgan-Voyce 	&$C_n(x)$ 	   	     &Morgan-Voyce	    & $B_n(x)$ 	 &  $(x+2+\sqrt{x^2+4x})/2$ & $(x+2-\sqrt{x^2+4x})/2$  \\
   Vieta-Lucas 		&$v_n(x)$ 	   	     &Vieta	            & $V_n(x)$ 	 &  $(x+\sqrt{x^2-4})/2$    & $(x-\sqrt{x^2-4})/2$     \\   \hline
\end{tabular}}
\end{center}
\caption{$\Lt{n}(x)$ and its conjugate $\Ft{n}(x)$.} \label{Equivalent}
\end{table}

For the rest of this paper we suppose that $\deg\left(d \right)>\deg\left(g \right)$. For instance, the familiar examples in Tables \ref{familiarfibonacci} and \ref{Equivalent}  
satisfy this condition. Notice that Jacobsthal and Jacobsthal-Lucas polynomials defined as 
$j_{n}(x)= j_{n-1}(x)+2x j_{n-1}(x)$ are GFP but they do not satisfy the mentioned condition. So, we do not study those polynomials here in this paper. 

\section{Fibonacci type polynomial irreducibility} \label{FibonacciSection}

In this section we discuss the irreducibility and reducibility of GFP of Fibonacci type. In particular, we give a complete classification (reducible and irreducible) for the familiar  
polynomials of Fibonacci type given in Table \ref{familiarfibonacci}. In the end of the section we give a more general theorem to determine whether a GFP of Fibonacci type is irreducible. 

The following lemma generalizes \cite[Lemma 5]{HoggattLong}. The proof can be done by induction, so we omit it. 

\begin{lemma}\label{genHogattLemma}
$\Ft{n}(x)= \sum _{i=0}^{\lfloor \frac{n-1}{2}\rfloor } \binom{n-i-1}{i} d(x)^{n-2 i-1}g(x)^i $.
\end{lemma}

Hogatt et al. \cite{HoggattLong} defined the bivariate  generalized Fibonacci polynomial  
$$u_n(x,y)=xu_{n-1}(x,y) +y u_n(x,y)\quad \text{ with } u_0(x,y)=0 \quad \text{ and } \quad  u_n(x,y)=1.$$ 
In their version of Lemma \ref{genHogattLemma} for $u_n(x,y)$ it holds that $u_n(x,y^2)$ is a homogeneous polynomial. 
Webb et al. \cite{WebbParberry} proved that $u_p(x,1)$ is irreducible over $\mathbb{Q}$ if and only if $p$ is a prime number.  These two results were used in 
\cite{HoggattLong} to prove that $u_p(x,y)$ is irreducible over $\mathbb{Q}$ if and only if $p$ is a prime number.  However, we need some  caution on the 
interpretation of these results. For example, in the result proved by Webb we cannot substitute $x$ by any polynomial.  Thus, if instead of $x$ 
we take $x^3$ we obtain that  $u_3(x^3,1)=(x+1)(x^2-x+1)$, so this new polynomial is reducible. Similarly, we can construct  examples to  
show that  $u_p(x,y)$ is not always irreducible for every prime and for every choice of $y$. For instance, if instead of $y$ we take $-y^{2k}$,  
it holds that  $u_p(x,-y^{2k})$ is not always irreducible when $p$ is a prime number, with $k\ge 0$. For example,  
$u_5(x,-y^2)=(-x^2-x y+y^2)(-x^2+x y+y^2)$.  In general, this gives a factoring for Chebyshev  polynomials of second kind $U_p(x)$.  
Thus,  if $p=2k+1$, then $U_p(x)= (U_{k+1}(x)-y^kU_k(x))(U_{k+1}(x)+y^kU_k(x))$ (see Proposition \ref{ChebyshevSecond}). Some other examples, in which  
$u_p(x,y)$ is reducible,  occurs taking  $y=-1, -4,-5, -9, -20$.  In particular,  $u_5(x,-5)=(x^2-5 x+5)(x^2+5 x+5)$  and  
$u_5(x+2,-1)=(x^2+3 x+1)(x^2+5 x+5)$.

We now recall the first of our main questions in this paper. Is it true that $\Ft{p}(x)$ is irreducible if and only if $p$ is prime? 
From the above discussion and Proposition \ref{ChebyshevSecond}, we can see some counterexamples to determine that the question  is 
not true in general. Since there are some families of the generalized Fibonacci polynomial that are irreducible if and only if $p$ is a prime number,  
the question is still valid.  In this section we explore the question for families of GFP of the Fibonacci type.  
(From the definition \eqref{Fibonacci;general:FT},  we know that families of GFP of Fibonacci type depend on their initial conditions.)  
Thus, we reformulate the question as,  under what conditions of $d(x)$ and $g(x)$ are the families of GFP of the Fibonacci type is irreducible when $p$ is a prime number. 

Note that from \cite[Proposition 6]{florezHiguitaMuk2018} we know that  $\Ft{n}(x)$ is reducible if $n$ is a composite number. For the remaining part of the paper  
we use $F_n(x)$ to denote the classic Fibonacci polynomial as defined in the introduction. 

\begin{lemma}[\cite{WebbParberry}] \label{FiboIrreduWebbParberry}  The Fibonacci polynomial $F_p(x)$ is irreducible over $\mathbb{Q}$ if and only if $p$ is a prime number. 
\end{lemma}

\begin{proposition} \label{ChebyshevSecond} Let $m(x)$ be a polynomial in $\mathbb{Z}[x]$ and let $p$ be an odd number. If $g(x)=-m(x)^2$,  then $\Ft{p}(x)$ is reducible. 
\end{proposition}

\begin{proof} If $p = 2k+1$, from  \cite{florezHiguitaMuk2018} or \cite[Proposition 1]{FlorezMcAnallyMuk} we have that 
$\Ft{p}(x)=\Ft{k+1}^2(x)+g(x) \Ft{k}^2(x)$. Since $g(x)=-m(x)^2$ the conclusion follows. 
\end{proof}

The previous proposition shows that  Chebyshev polynomials of second kind, Morgan-Voyce polynomials and Vieta polynomials are reducible over $\mathbb{Q}$ when $p$ is a prime number. 

\begin{proposition} \label{FibonacciIrreducibleDegOne} If $g(x)=1$ and  $d(x)=ax+b$ with $a\ne 0$, then $\Ft{p}(x)$ is irreducible over $\mathbb{Q}$. 
\end{proposition}

\begin{proof} First of all we observe that if $g(x)=1$ and  $d(x)=ax+b$, then $\Ft{p}(x)=(F_p\circ d)(x)$. Since  both $F_p(x)$ and $ax+b$ are irreducible, we have that $\Ft{p}(x)$ is irreducible.
\end{proof}

\begin{lemma} \label{FibonacciIrreducibleLemma} If $g(x)=k$ and  $d(x)=ax$ with $a\ne 0$, and $k\in \mathbb{Z}_{>0}$, then $\Ft{p}(x)$ is irreducible in $\mathbb{Q}[x]$. 
\end{lemma}

\begin{proof} Since  both $F_p(x)$ and $ax$ are irreducible over $\mathbb{Q}$, we have that $\Ft{p}^{*}(x):=(F_p\circ d)(x)$ is irreducible.

To complete  this proof we need the following lemma. This lemma is an adaptation, to what need here, of a result that is well-known in the literature (see for example \cite{Bonciocat,Cafure}). 

\textbf{Lemma.} Let $f(x)$ be a polynomial of degree $n$ with $f(0)\ne0$. Then $f(x)$ is irreducible if and only if $t^n f(1/t)$ is irreducible. 

This lemma implies that 
$$\Ft{p}^{*}(x) \text{ is irreducible }  \iff      s(t):= (k^{1/2}t)^{p-1} \Ft{p}^{*}\left(\frac{1}{k^{1/2}t}\right) \text{ is irreducible }.$$
Therefore,
$$ s(t) \text{ is irreducible.}       \iff h(r):= (r)^{p-1} s\left(\frac{1}{r}\right) \text{ is irreducible.}$$

Taking $g(x)=k$ and  $d(x)=ax$ with $a\ne 0$ and $k\in \mathbb{Z}_{>0}$, we obtain $\Ft{p}(x)$. This and Lemma \ref{genHogattLemma}, imply that 
$h(x)=\Ft{p}(x)$.  
\end{proof}

Propositions \ref{ChebyshevSecond} and \ref{FibonacciIrreducibleDegOne} show whether or not the familiar polynomials of Fibonacci type are irreducible (see Table \ref{familiarfibonacci}), when $p$ is prime.   

\begin{lemma} [\cite{hoggattRoots,WebbParberry}]\label{FibonacciIrreducibleRoots}  Let $i:=\sqrt{-1}$ and let 
$\gamma_m= 2i \cos\frac{j \pi}{p}$ for $j=1,2, \dots, p-1$, where $p$ is a prime number. Then $\Gamma=\{\gamma_1, \dots, \gamma_{p-1}\}$ are the  roots of the Fibonacci polynomial $F_p(x)$. 
\end{lemma}

The following lemma is a generalization of Capelli's (see Lemma \ref{CapelliLemma}) to what we need here in this paper.  

\begin{lemma}\label{JCCapelliLemma} 
Let $f(x)$, $g(x)$, $h(x)$ in $K[x]$, with $K$ a field, $\deg h(x)>\deg g(x)$, and $f(x)=a_nx^n+a_{n-1}x^{n-1}+\ldots+a_0$, where $a_0\ne 0$. 
Let $\gamma$ be any root of $f(x)$ in an algebraic closure of $K$. The polynomial 
$p(x)=a_nh(x)^n+a_{n-1}g(x)h(x)^{n-1}+a_{n-2}g(x)^2h(x)^{n-2}+\ldots+a_0g(x)^n$,  
 is irreducible over $K$ if and only if $f(x)$ is irreducible over $K$ and $h(x)-g(x)\gamma$ is irreducible over $K(\gamma)$.
\end{lemma}

\begin{proof}
Let $\theta$ be a root of $h(x)-g(x)\gamma$ in the algebraic closure of $K$. So, $h(\theta)=g(\theta)\gamma$. If $h(\theta)=0$, then $g(\theta)=0$, 
since $\gamma\ne 0$. Therefore, $p(\theta)=0$. If $h(\theta)\neq 0$, then we have
\begin{equation*}
p(\theta)=g(\theta)^nf\left(\frac{h(\theta)}{g(\theta)}\right)=g(\theta)^nf(\gamma)=0.
\end{equation*}
Notice that $\deg p(x)=n\deg h(x)$. Also, we have
\begin{equation*}
[K(\theta):K]=[K(\theta):K(\gamma)][K(\gamma):K].
\end{equation*}

Since $\deg h(x)>\deg g(x)$, we have $\deg(h(x)-g(x)\gamma)=\deg h(x)$. Therefore, this gives that $[K(\theta):K(\gamma)]\leq \deg h(x)$ and  
$[K(\gamma):K]\leq n$. Thus, $p(x)$ is irreducible over $K$ if and only if $[K(\theta):K]=n\deg h(x)$, which is only the case if and only if  
$[K(\theta):K(\gamma)]=\deg h(x)$ and $[K(\gamma):K]=n$. This holds if and only if $f(x)$ is irreducible over $K$ and $h(x)-g(x)\gamma$ is 
irreducible over $K(\gamma)$.
\end{proof}

\begin{theorem}\label{JCFibonacciIrreducibleCapelli}
Let $p>2$ be a prime number and let $\Gamma=\{\gamma_1, \dots, \gamma_{p-1}\}$ be the set of  roots of  $F_p(x)$. A GFP $\Ft{p}(x)$ is irreducible over $\mathbb{Q}$ if and only if $d(x)^2-g(x)\gamma^2$ is irreducible over $\mathbb{Q}\left(\gamma^2\right)$ for some $\gamma\in\Gamma$.
\end{theorem}

\begin{proof}
For all $z\in\mathbb{C}$ such that $g(z)\neq 0$, we can deduce that
\begin{equation*}
\Ft{p}(z)=g(z)^{\frac{p-1}{2}}F_p\left(\frac{d(z)}{g(z)^{1/2}}\right).
\end{equation*}
From Lemma \ref{genHogattLemma} we know that $F_n(x)\in\mathbb{Z}[x^2]$. So, we let $G_p(x)$ be a polynomial in $\mathbb{Z}[x]$ such that $G_p(x^2)=F_p(x)$.  
Since $F_p(x)$ is irreducible over $\mathbb{Q}$, it follows that $G_p(x)$ is irreducible over $\mathbb{Q}$. For all $z\in\mathbb{C}$ such that $g(z)\neq 0$, we deduce  
that 
\begin{equation*}
\Ft{p}(z)=g(z)^{\frac{p-1}{2}}G_p\left(\frac{d(z)^2}{g(z)}\right). 
\end{equation*}
 The conclusion follows from Lemma \ref{JCCapelliLemma}.
\end{proof}

\begin{lemma} [\cite{Tschebotarow}] \label{CapelliLemma}  
Let $f(x)$, $r(x)$ in $K[x]$, where $K$ is a field.  
Let $\gamma$ be any root of $f(x)$ in an algebraic closure of $K$. The composition 
$f(r(x))$ is irreducible over $K$ if and only if $f(x)$ is irreducible over $K$ and
$r(x)-\gamma$  is irreducible over $K(\gamma)$.
\end{lemma}

\begin{corollary} \label{FibonacciIrreducibleCapelli} Let 
$\Gamma=\{\gamma_1, \dots, \gamma_{p-1}\}$ be the set of  roots of the Fibonacci polynomial $F_p(x)$, where $p$ is prime number. Let $g(x) \in \mathbb{Z}_>0$. The polynomial $\Ft{p}(x)$ is irreducible in $\mathbb{Q}$ if and only if 
$d(x)-\gamma$ is irreducible over $\mathbb{Q}(\gamma)$  for some $\gamma \in \Gamma$.
\end{corollary}

\begin{proof} We consider the generalized Fibonacci polynomial $\Ft{p}^{*}(x)$ defined when  $g^{*}(x)$ is a constant and $d^{*}(x)=x$ 
(we use $^{*}$ to avoid any ambiguity with the upcoming analysis, using similar notation).  From Lemma \ref{FibonacciIrreducibleLemma} we have that $\Ft{p}^{*}(x)$ is irreducible. 

Now consider the generalized Fibonacci polynomial 
$\Ft{p}(x)$ where $g(x)$ is a positive constant (integer) and $d(x)$ is a polynomial that satisfies that $d(x)-\gamma$ is irreducible over $\mathbb{Q}(\gamma)$  for some $\gamma \in \Gamma$. Note that $\Ft{p}(x)$ is the composition of $\Ft{p}^{*}(x)$ with $d(x)$, i.e.  $\Ft{p}(x)=(\Ft{p}^{*} \circ d)(x)$. 
This and Lemma \ref{CapelliLemma}, imply that $\Ft{p}(x)$ is irreducible if and only if $\Ft{p}^{*}(x)$ is irreducible. 
\end{proof}

As a corollary of the previous results we have that if $\Ft{n}$ satisfies any of the conditions given in Propositions \ref{ChebyshevSecond}, \ref{FibonacciIrreducibleDegOne},  Theorem  \ref{JCFibonacciIrreducibleCapelli}, and Corollary \ref{FibonacciIrreducibleCapelli},  we have this. 
Suppose that the prime decomposition of $n$ is given by $n=p_1^{n_1} p_2^{n_2}\cdots p_k^{n_s}$, where  $p_1, p_2,\cdots, p_s$ are distinct odd primes. Then   $\Ft{p_i}$ is an irreducible factor of $\Ft{n}$ for $i=1,2, \dots, s$. The proof of this fact follows straightforwardly using Propositions \ref{ChebyshevSecond}, \ref{FibonacciIrreducibleDegOne},  Theorem  \ref{JCFibonacciIrreducibleCapelli}, Corollary \ref{FibonacciIrreducibleCapelli}, and \cite[Proposition 6]{florezHiguitaMuk2018}.

\section{Lucas type polynomials irreducibility} \label{LucasSection}

In this section we discuss the irreducibility of GFP of Lucas type. In particular, we show that the familiar polynomials of Lucas type given in Table \ref{familiarfibonacci} are irreducible.  
In the end of the section we give a more general theorem to determine whether a GFP of Lucas type is irreducible. 

Lemma \ref{genHogattLucasLemma} and Proposition  
\ref{genHogattLucasProp} are generalizations of Bergum and Hoggatt  results in \cite{HoggattBergum}. The proof of both cases follows  
by a natural adaptation of their proof to the GFP of Lucas type given in this paper. 

\begin{lemma}\label{genHogattLucasLemma} If $\Lt{n}(x)$ is the conjugate of $\Ft{n}(x)$, then 
\[\Lt{n}(x)=\frac{1}{\alpha} \sum _{i=0}^{\lfloor \frac{n}{2}\rfloor }  \frac{n}{n-i} \binom{n-i}{i} d(x)^{n-2 i}g(x)^i.\]
\end{lemma}

\begin{proof}
From \cite[Proposition 3, Part 2]{FlorezMcAnallyMuk} we obtain that $\alpha  \Lt{n}(x)=g(x) \Ft{n-1}(x)+\Ft{n+1}(x)$. This and Lemma \ref{genHogattLemma} gives that 
\begin{eqnarray*}
\alpha  \Lt{n}(x) &=&(g(x) \Ft{n-1}(x)+\Ft{n+1}(x)\\
			     &=&  \sum_{i=1}^{\lfloor n/2 \rfloor} \left({ n-i-1 \choose i-1}+{ n-i \choose i} \right)d(x)^{n-2i} g(x)^i + d(x)^{n} .
\end{eqnarray*}
This completes the proof. 
\end{proof}

Bergum and Hoggatt \cite{HoggattBergum} defined a  bivariate generalized Lucas polynomial by $v_n(x,y)=x v_{n-1}(x,y) +y v_n(x,y)$ with $v_0(x,y)=2$ and   
$v_1(x,y)=x$. They proved that the Lucas polynomials $v_n(x,1)$ are irreducible over $\mathbb{Q}$ if and only if $n$ is power of $2$, and proved
that the polynomials  $v_p(x,1)/x$ are irreducible over $\mathbb{Q}$ if and only if $p>2$ is prime. In their version of Lemma \ref{genHogattLucasLemma}  
for $v_n(x,y)$  it holds that $v_n(x,y^2)$ are homogeneous polynomials. This implies that $v_n(x,y)$ are irreducible over $\mathbb{Q}$ if and only if  
$n$ is a power of 2; and that $v_p(x,y)/x$ is irreducible over $\mathbb{Q}$ if and only if $p>2$  is a prime number. Again, we need some caution on  
the interpretation of these results. Thus, if instead of $x$ we take $x^2+x$ we obtain that  $v_3(x^2+x,1)/(x^2+x)=(x^2-x+1) (x^2+3 x+3)$. Similarly,  
we can construct examples to show that  $v_{n}(x,y)$ is not always irreducible for $n$ a power of $2$ and for every choice of $y$. For instance, if instead  
of $y$ we take $-y^{2k}$, with $k$ even, it holds that $v_2(x,-2y^{2k})=(x-2 y^k) (x+2 y^k)$. Similar results hold when $y$ is replaced by $-(t/2) y^{2k}$  
where $t$ is an even perfect square.  Another example is $v_4(x,x-2)=(x^2-2)(x^2+4 x-4)$. Examples, to show that $v_p(x,y)/x$ is not always irreducible  
for every choice of $y$ and for every  prime, can be constructed by replacing $y$ by $-p y^2$ in $v_{p}(x,y)$ with $p=3 \bmod 4$. For instance, 
$v_3/x=(x-3 y) (x+3 y)$, $v_7/x=(x^3+7 x^2 y-49 y^3) (x^3-7 x^2 y+49 y^3)$, and $v_{11}/x=(x^5+11 x^4 y-363 x^2 y^3-1331 x y^4-1331 y^5)(x^5-11 x^4 y+363 x^2 y^3-1331 x y^4+1331 y^5)$. Taking these examples, from the point of view of GFP of Lucas type, states as, for a fixed prime $q\equiv 3 \bmod 4$, and picking $g(x)=-q$, then $\Lt{q}(x)/p_1(x)$ is reducible over $\mathbb{Q}$, (see Proposition \ref{JCSpecialCoroCaseLucas}). 

All the above examples (in Section \ref{FibonacciSection} and in Section \ref{LucasSection}) show that there is not clarity on both the quantifiers and the initial 
conditions on the results in \cite{HoggattBergum,HoggattLong,WebbParberry}. 
So, one of the motivations for this paper is to revisit some of the main results given by Hoggatt, Bergum, Long, Parberry, and Webb in the mentioned papers, 
and then use those results to give more general theorems.  

We recall again our second main question in this paper. Is it true that $\Lt{p}(x)/p_1(x)$ is irreducible if and only if $p>2$ is a prime number? 
From the above discussion, we can see some counterexamples to determine that the question is not true in general. Again, since there are some  
families of the generalized Fibonacci polynomial that are irreducible if and only if $p$ is prime, the question is still valid. 

In this section we explore the question for families of GFP of Lucas type $\Lt{p}(x)$.  Thus, in this section we explore the conditions that we have to impose on 
$p_{1}(x)$, $d(x)$ and $g(x)$ to obtain   $\Lt{p}(x)/p_1(x)$ is irreducible, when $p$ is a prime number.  Note that from \cite[Proposition 7]{florezHiguitaMuk2018} 
we know that $\Lt{n}(x)/p_1(x)$ is reducible if $n$ is a  composite number with an odd divisor.

Proposition \ref{genHogattLucasProp} gives enough conditions to prove whether the polynomials $\Lt{q}(x)/p_1(x)$ of Lucas type of the form as 
 shown in Table \ref{familiarfibonacci} are irreducible when $q$ is prime. The proof uses  the Eisenstein  criterion \cite{Lang}.
  
\begin{proposition}\label{genHogattLucasProp} Let $q>2$ be a prime number, with $\gcd(q,g(x))=1$.

\begin{enumerate}

\item \label{genHogattLucasPart1}  If $d(x)=a x^t$, with $\gcd(q,a)=1$, then  $\Lt{q}(x)/p_1(x)$ is irreducible over $\mathbb{Q}$.

\item \label{genHogattLucasPart2} If $d(x)=c x+b$, then  $\Lt{q}(x)/p_1(x)$ is irreducible over $\mathbb{Q}[x]$.

\end{enumerate}
\end{proposition}

\begin{proof}  

Proof of Part \ref{genHogattLucasPart1}. From\ Lemma \ref{genHogattLucasLemma} and the fact 
that $d(x)=\alpha p_1(x)$ we have that  
\[\Lt{q}(x)/p_1(x)=\sum _{i=0}^{\frac{q-1}{2}}  \frac{q}{q-i} \binom{q-i}{i} d(x)^{q-2 i-1}g(x)^i.\]
It is well known that  $\sum_{k=i}^{\frac{q-1}{2}}{ q \choose 2k}{ k \choose i}2^{-q+2i+1}=  \frac{q}{q-i} \binom{q-i}{i} $ 
(see for example, \cite{HoggattBergum,Gould, Riordan}). 
Since $q | { q \choose 2i}$ for $i=1,2, \dots, (q-1)/2$, we have that $q$ divides $\frac{q}{q-i} \binom{q-i}{i}$. This implies that $q$ divides 
$\frac{q}{q-i} \binom{q-i}{i} a^{q-2 i-1}g(x)^i$ for $1\le i\le (q-1)/2$. Since $\frac{q}{q-i} \binom{q-i}{i} g(x)^{q-2 i-1}=qg(x)^{q-2 i-1}$, when $i=(q-1)/2$ and  $\gcd(q,g(x))=1$, 
we have that $q^2$ does not divide the independent term of $\Lt{q}(x)/p_1(x)$. The fact that $\gcd(q,a)=1$, gives that $q$ does not divide $a^{q-1}$, the leading coefficient of $\Lt{q}(x)/p_1(x)$. These and  the Eisenstein  criterion complete the proof.

Proof of Part \ref{genHogattLucasPart2}. Let $G_q(x)$ be equal to $\Lt{q}(x)/p_1(x)$ as defined in Part \ref{genHogattLucasPart1} with 
$d(x)=x$ and let $H_q(x)$ be equal to $\Lt{q}(x)/p_1(x)$ as defined in Part \ref{genHogattLucasPart2}. Note that the composition of $G_q(x)$  
with $c x+b$, gives $H_q(x)$. Therefore, $H_q(x)=G_q(ax+b)$ is irreducible if and only if $G_q(x)$ is irreducible. Since $G_q(x)$ is irreducible 
for $q>2$, this completes the proof. 
\end{proof}

As a corollary of the previous proposition we obtain that the following polynomials (from Table \ref{familiarfibonacci}) are irreducible when $p$ is an odd prime number: 
Lucas $D_p(x)$; Pell-Lucas $Q_p(x)$; Fermat-Lucas $\vartheta_p(x)$; Chebyshev first kind $T_p(x)$; 
Morgan-Voyce $C_p(x)$; Vieta-Lucas $v_p(x)$. 

The following two propositions are known as Sch\"onemann and the Eisenstein criterions, respectively.  

\begin{proposition}[\cite{Bonciocat,Schonemann}]\label{SchonemannProposition} 
Let $q$ be a prime number. If $f(x)\in \mathbb{Z}[x]$ has the form $f(x)=g(x)^n+q m(x)$ with $g(x)$ an irreducible polynomial in $\mathbb{F}_q[x]$ 
and does not divide $m(x) \bmod{ q}$,  then $f(x)$ is irreducible.
\end{proposition}

\begin{proposition}[\cite{Bonciocat,Lang}]\label{EisensteinProposition} 
Let $q$ be a prime number and let $f(x)=a_nx^n+ ... + a_1x+a_0$ be a  polynomial in $\mathbb{Z}[x]$.  If  
$a_n \bmod q \ne 0$, $a_0 \bmod q^2 \ne 0$, and $a_i \bmod q = 0$ for $i=0,1, \dots, n-1$, then $f(x)$ is irreducible over $\mathbb{Q}$. 
\end{proposition}

\begin{proposition}\label{genHogattLucasProposition} Let $q>2$ be a prime number. 
If $d(x)\in\mathbb{Z}[x]$ is irreducible $\bmod\; q$, then $\Lt{q}(x)/p_1(x)$ is irreducible over $\mathbb{Q}$.
\end{proposition}

\begin{proof} Since $q$ is a factor of  $\frac{q}{n-i} \binom{q-i}{i}$ for  
$i=1, \dots, (q-1)/2$, from  Lemma  \ref{genHogattLucasLemma} we have that there is a  polynomial $h(x)\in \mathbb{Z}[x]$ such that 
$q h(x)=\sum _{i=1}^{\frac{q-1}{2}}  \frac{q}{n-i} \binom{q-i}{i} d(x)^{q-2 i-1}g(x)^i$.  
This and Lemma  \ref{genHogattLucasLemma} imply that 
\begin{equation}\label{EisensteinCoro}
\Lt{q}(x)/p_1(x)=d(x)^{q-1}+q h(x). 
\end{equation}
This decomposition and the fact that  $\deg(g(x))< \deg(d(x))$, imply that $\deg(h(x)) <\deg(d(x)^{q-1})= \deg(\Lt{q}(x)/p_1(x))$.

If we let $q t(x):=\sum _{i=1}^{\frac{q-1}{2}-1}  \frac{q}{n-i} \binom{q-i}{i} d(x)^{q-2 i-2}g(x)^i$, then $h(x)$ can be written in the form
$h(x)=  d(x)t(x)+g(x)^{\frac{q-1}{2}}$. This, the irreducibility of $d(x) \bmod q$,  $\deg(g(x))< \deg(d(x))$, and the fact that $\gcd(d(x),g(x))=1$, imply that
$\gcd(d(x),h(x))=\gcd(d(x),g(x))=1 \bmod q$. The desired conclusion follows by Proposition \ref{SchonemannProposition}.
\end{proof}

\begin{proposition}\label{genHogattLucasProp2} Let $q>2$ be a prime number and  let $d(x)$ be $a_kx^k+a_{k-1}x^{k-1}+\dots+a_1x+ a_0$ where $a_{k} \bmod q\ne 0$, and $a_{i} \bmod q= 0$ for $i=0, 1, \dots, k-1$. If $g(0) \bmod q \ne 0$, then $\Lt{q}(x)/p_1(x)$ is irreducible over $\mathbb{Q}$.
\end{proposition}

\begin{proof} Since $a_{i} \bmod q= 0$ for $i=0, 1, \dots, k-1$, we have that there is a polynomial $p(x) \in \mathbb{Z}[x]$ such that  
$d(x)^{q-1} =(a_kx^k+a_{k-1}x^{k-1}+\dots+a_1x+ a_0)^{q-1}= a_k^{q-1} x^{k(q-1)}+q p(x)$. (Note that $p(x)$ can be zero.) This and 
Lemma  \ref{genHogattLucasLemma} imply that 
\[\Lt{q}(x)/p_1(x)=\sum _{i=1}^{\frac{q-1}{2}-1}  \frac{q}{n-i} \binom{q-i}{i} d(x)^{q-2 i-1}g(x)^i+(a_k^{q-1} x^{k(q-1)}+q p(x))+q g(x)^{\frac{q-1}{2}}.\]
Since $q$ is a factor of  $\frac{q}{n-i} \binom{q-i}{i} d(x)^{q-2 i-1}g(x)^i$ for every for  $i=1, \dots, (q-1)/2-1$, we have that there is a  
polynomial $h(x)\in \mathbb{Z}[x]$ such that $q h(x)=\sum _{i=1}^{\frac{q-1}{2}-1}  \frac{q}{n-i} \binom{q-i}{i} d(x)^{q-2 i-1}g(x)^i$.  (Note that  $q\mid h(0)$ and $q\mid p(0)$,  because $q\mid a_0$.)
Therefore, 
\begin{equation}\label{EisensteinCoro}
\Lt{q}(x)/p_1(x)=q h(x)+a_k^{q-1} x^{k(q-1)}+q p(x)+q g(x)^{\frac{q-1}{2}}. 
\end{equation}
This decomposition of $\Lt{q}(x)/p_1(x)$ and the fact that  
$\deg(g(x))< \deg(d(x))$, imply that $\lc(\Lt{q}(x)/p_1(x)) =a_k^{q-1}$. This and \eqref{EisensteinCoro} prove that $q$ 
divides all coefficients of $\Lt{q}(x)/p_1(x)$ except its leading coefficient.

To complete the proof using the Eisenstein criterion (Proposition \ref{EisensteinProposition}), we prove that $q^2$ does not divide the independent term of $\Lt{q}(x)/p_1(x)$. 
From \eqref{EisensteinCoro} we can see that the independent coefficient of $\Lt{q}(x)/p_1(x)$ has the form $q \left(h(0)+p(0)+g(0)^{\frac{q-1}{2}}\right)$. This and the fact that $q$ does not divide $g(0)$ implies that $q^2$ does not divide $q \left(h(0)+p(0)+g(0)^{\frac{q-1}{2}}\right)$. This completes the proof.
\end{proof}

The statement of the previous proposition can be generalized to this with the same proof. Let $q>2$ be a prime number and  let $b_kx^k+b_{k-1}x^{k-1}+\dots+b_1x+ b_0$ 
be the coefficients of $d(x)^{q-1}+ qg(x)^{\frac{q-1}{2}} $  where $b_{k} \bmod q\ne 0$, $b_{0} \bmod q^2\ne 0$,  and $b_{i} \bmod q= 0$ for $i=0, 1, \dots, k-1$. Then $\Lt{q}(x)/p_1(x)$ is irreducible over $\mathbb{Q}$.

In this corollary we give a partial irreducible decomposition of $\Lt{n}$ when $n$ a composite number. 

\begin{corollary} \label{PrimeFactorLucas} 
Let $n=p_1^{n_1} p_2^{n_2}\cdots p_k^{n_k}$ be the prime decomposition of $n$, where  $p_1, p_2,\cdots, p_k$ are distinct odd prime numbers.
\begin{enumerate}
\item  If $d(x)=a x^t$ with $\gcd(p_i,a)=\gcd\left(p_i,g(x)\right)=1$, then $\Lt{p_i}/d(x)$  is an irreducible factor of  $\Lt{n}$.

\item If $d(x)=c x+b$ with $\gcd\left(p_i,g(x)\right)=1$, then $\Lt{p_i}/d(x)$  is an irreducible factor of  $\Lt{n}$.

\item If $d(x)\in\mathbb{Z}[x]$ is irreducible $\bmod\; p_i$, then $\Lt{p_i}/d(x)$  is an irreducible factor of  $\Lt{n}$.

\item If $g(0) \bmod p_i \ne 0$ and $d(x)=a_kx^k+a_{k-1}x^{k-1}+\dots+a_1x+ a_0$ where $a_{k} \bmod p_i\ne 0$, and $a_{i} \bmod p_i= 0$ for $i=0, 1, \dots, k-1$, then $\Lt{p_i}/d(x)$  is an irreducible factor of  $\Lt{n}$.
\end{enumerate}
\end{corollary}

\begin{proof} It follows straightforwardly using \cite[Corollary 2]{florezHiguitaMuk2018} and Propositions \ref{genHogattLucasProp},  \ref{genHogattLucasProposition}, and \ref{genHogattLucasProp2}.
\end{proof}

The proof of Parts \ref{Power2HogattLucasPart1} and \ref{Power2HogattLucasPart2} of the following proposition follows, again, by the Eisenstein criterion and Lemma \ref{genHogattLucasLemma} where $p=2$. The proof of Part \ref{Power2HogattLucasPart3} is similar to the proof of Proposition \ref{genHogattLucasProposition}.  So, we omit details. 

\begin{proposition}\label{genHogattLucasPower2} Let $n=2^k$ for $k\ge1$.  

\begin{enumerate}

\item \label{Power2HogattLucasPart1}  If $d(x)=a x^t$ with $a$ and $g(x)\not\equiv 0 \bmod 2$ odd integers, then $\Lt{2^t}(x)$ is irreducible over $\mathbb{Q}$ for $t\ge 1$.

\item \label{Power2HogattLucasPart2} If $d(x)=c x+b$ with $c$ and $g(x)$ odd integers, then  $\Lt{2^t}(x)$ is irreducible over $\mathbb{Q}$ for $t\ge 1$.

\item \label{Power2HogattLucasPart3} If $d(x)\in\mathbb{Z}[x]$ is irreducible $\bmod\; 2$, then $\Lt{2^t}(x)$ is irreducible over $\mathbb{Q}$ for $t\ge 1$.

\item \label{Power2HogattLucasPart4} If $g(0)$ is an odd integer and $d(x)=a_kx^k+a_{k-1}x^{k-1}+\dots+a_1x+ a_0$ where $a_{k}$ is an odd integer and $a_{i}$ is a even integer for $i=0, 1, \dots, k-1$, then $\Lt{2^t}(x)$ is irreducible over $\mathbb{Q}$ for $t\ge 1$.

\end{enumerate}
\end{proposition}

\begin{lemma} [\cite{hoggattRoots}]\label{FibonacciIrreducibleRoots}  Let $p=2k+1$ be a prime number, let $i:=\sqrt{-1}$ and let 
$\tau_j=\pm 2i \sin\frac{j \pi}{p}$ for $j=1,2, \dots, k$. Then $T=\{\tau_1, \dots, \tau_{p-1}\}$ are the  roots of $L_p(x)/x$. 
\end{lemma}

\begin{theorem} \label{JCLucasIrreducibleCapelli} 
Let $p>2$ be a prime number and let $T=\{\tau_1, \dots, \tau_{p-1}\}$ be the set of  roots of $L_p(x)/x$, where $L_p(x)$ is the Lucas polynomial. The polynomial  
$\alpha\Lt{p}(x)/d(x)=\Lt{p}(x)/p_1(x)$ is irreducible over $\mathbb{Q}$ if and only if $d(x)^2-g(x)\tau^2$ is irreducible over $\mathbb{Q}\left(\tau^2\right)$ for some $\tau\in T$.
\end{theorem}

\begin{proof} This proof is similar to the proof of Theorem \ref{JCFibonacciIrreducibleCapelli} (see also the proof of Theorem  \ref{JCLucasIrreducibleCapelli2m}). Using Lemma \ref{genHogattLucasLemma} instead of Lemma \ref{genHogattLemma}. Then setting $J{p}(x):=L_p(x)/x$ and then replacing $F_p(x)$ 
in the proof of Theorem \ref{JCFibonacciIrreducibleCapelli} by $J{p}(x)$, we obtain the desired result. 
\end{proof}

\begin{corollary} \label{LucasIrreducibleCapelli} Let $g(x)=1$ and let 
$T=\{\tau_1, \dots, \tau_{p-1}\}$ be the set of  roots of  $L_p(x)/x$ where $L_p(x)$ is the Lucas polynomial. The polynomial $\alpha\Lt{p}(x)/d(x)=\Lt{p}(x)/p_1(x)$ is irreducible in $\mathbb{Q}$ if and only if 
$d(x)-\tau$ is irreducible over $\mathbb{Q}(\tau)$  for some $\tau \in \Gamma$.
\end{corollary}

\begin{proof} Consider $\Lt{p}(x)/d(x)$ with $g(x)=1$ and $d(x)$ a polynomial that satisfies that $d(x)-\tau$ is irreducible over $\mathbb{Q}(\tau)$  for some $\tau \in T $. Note that $\Lt{p}(x)/d(x)$ is the composition of $L_p(x)/x$ with $d(x)$, i.e.  $\Lt{p}(x)=(L_p \circ d)(x)$. 
This and Lemma \ref{CapelliLemma}, imply that $\Lt{p}(x)/d(x)$ is irreducible if and only if $L_p(x)/x$ is irreducible. 
\end{proof}

\begin{proposition} \label{JCSpecialCoroCaseLucas} If $g(x)=-qh(x)^2$, where  $h(x)\in\mathbb{Z}[x]$ and $q\equiv 3 \bmod 4$, then $\Lt{q}(x)/p_1(x)$ is reducible over $\mathbb{Q}$.
\end{proposition}

\begin{proof}
By Theorem \ref{JCLucasIrreducibleCapelli}, we only need to show that $d(x)^2+qh(x)^2\tau_j^2$ is reducible over $\mathbb{Q}\left(\tau_j^2\right)$ for some 
$\tau_j \in T$ (see Lemma \ref{FibonacciIrreducibleRoots}). Since 
$$d(x)^2+q\tau_j^2h(x)^2=\big(d(x)-\sqrt{q}i\tau_j h(x)\big)\big(d(x)+\sqrt{q}i\tau_j h(x)\big),$$ 
it suffices to show that $\sqrt{q}i\tau_j\in\mathbb{Q}\left(\tau_j^2\right)$. From Lemma \ref{FibonacciIrreducibleRoots}, we conclude that $\sqrt{q}i\tau_j=\pm 2\sqrt{q}\sin\frac{j \pi}{q}$. So, $\tau_j^2=-4\sin^2\frac{j \pi}{q}$. The fact that $\cos\frac{2j \pi}{q}=1-2\sin^2\frac{j \pi}{q}$, implies that $\mathbb{Q}\big(\tau_j^2\big)=\mathbb{Q}\big(\cos\frac{2j \pi}{q}\big)$. 
 
Since $\gcd(2j,q)=1$,  we have $\big[\mathbb{Q}\big(\cos\frac{2j \pi}{q}\big):\mathbb{Q}\big]=\frac{q-1}{2}$. So, we only need to show that 
$\sqrt{q}\sin\frac{j \pi}{q}\in \mathbb{Q}\big(\cos\frac{2j \pi}{q}\big)$. 
Thus, we know that $\cos\frac{2j \pi}{q}\in \mathbb{Q}\big(\sqrt{q}\sin\frac{j \pi}{q}\big)$, and therefore we just need to show that 
$\big[\mathbb{Q}\big(\sqrt{q}\sin\frac{j \pi}{q}\big):\mathbb{Q}\big]\leq\frac{q-1}{2}$. 

Since $q\equiv 3\bmod 4$, we have the quadratic Gauss sum
\begin{equation*}
\sum_{n=0}^{q-1}e^{\frac{2 \pi i n^2}{q}}=i\sqrt{q}.
\end{equation*}
So, 
\begin{equation*}
\sqrt{q}\sin\frac{j \pi}{q}=\frac{\big(e^{\frac{ij \pi}{q}}-e^{-\frac{ij \pi}{q}}\big)}{2}\sum_{n=0}^{q-1}e^{\frac{2 \pi in^2}{q}}=\frac{1}{2}\sum_{n=0}^{q-1}\big(\zeta^{n^2+j/2}-\zeta^{n^2-j/2}\big),
\end{equation*}
where $\zeta=e^{\frac{2i \pi}{q}}$. Since $\sin\frac{j \pi}{q}=\sin\frac{(q-j) \pi}{q}$, we may assume that $j$ is even. Then since the conjugates of $\zeta$ over $\mathbb{Q}$ are $\zeta^2,\zeta^3,\ldots,\zeta^{q-1}$, it therefore follows that the conjugates of $\sqrt{q}\sin\frac{j \pi}{q}$ over $\mathbb{Q}$ are
\begin{equation*}
\frac{1}{2}\sum_{n=0}^{q-1}\left(\zeta^{m\left(n^2+j/2\right)}-\zeta^{m\left(n^2-j/2\right)}\right),\quad m=1,2,\ldots,q-1.
\end{equation*}
Also notice, however, that
\begin{equation*}
\frac{1}{2}\sum_{n=0}^{q-1}\left(\zeta^{m\left(n^2+j/2\right)}-\zeta^{m\left(n^2-j/2\right)}\right)=\frac{1}{2}\sum_{n=0}^{q-1}\left(\zeta^{m(q-1)\left(n^2+j/2\right)}-\zeta^{m(q-1)\left(n^2-j/2\right)}\right)
\end{equation*}
for all $m\in\mathbb{Z}$ since $\zeta$ and $\zeta^{q-1}$ are complex conjugates. Hence, the number of conjugates of $\sqrt{q}\sin\frac{2j \pi}{q}$ over $\mathbb{Q}$ is at most $\frac{q-1}{2}$, and so $\big[\mathbb{Q}\big(\sqrt{q}\sin\frac{2j \pi}{q}\big):\mathbb{Q}\big]\leq\frac{q-1}{2}$. This completes the proof. 
\end{proof}

This previous proposition in combination with Proposition \ref{genHogattLucasProp2} give rise to infinite families of GFP of Lucas type that have special behavior. 
For example, if $d(x)=x$ and $g(x)=3$, we have that the Lucas polynomials with these conditions satisfy that $\Lt{3}(x)/d(x)=x^2-9=(x-3) (x+3)$ and that $\Lt{p}(x)/d(x)$ 
is irreducible for every prime $p\ne 3$; if $d(x)=x$ and $g(x)=7$, we have that the Lucas polynomials with these conditions satisfy that   
$\Lt{7}(x)/d(x)= x^6-49 x^4+686 x^2-2401=(x^3-7 x^2+49)(x^3+7 x^2-49)$ and that $\Lt{p}(x)/d(x)$ is irreducible for every prime $p\ne 7$; 
if $d(x)=x$ and $g(x)=11$, we have that the Lucas polynomials with these conditions satisfy that   
$\Lt{11}(x)/d(x)=x^{10}-121 x^8+5324 x^6-102487 x^4+805255 x^2-1771561=$ $(x^5-11 x^4+363 x^2-1331 x+1331)(x^5+11 x^4-363 x^2-1331 x-1331)$ 
and that $\Lt{p}(x)/d(x)$ is irreducible for every prime $p\ne 11$.

Since $d(x)=\alpha p_{1}(x)$, we have $\Lt{p}(x)/d(x)=\Lt{p}(x)/p_1(x)$ in  Theorem \ref{JCLucasIrreducibleCapelli} and Corollary \ref{LucasIrreducibleCapelli} when $\alpha =1$. 

\begin{theorem} \label{JCLucasIrreducibleCapelli2m}
Let $m\in\mathbb{N}$ and $R=\{\rho_1, \dots, \rho_{2^m}\}$ be the set of  roots of $L_{2^m}(x)$. The polynomial $\Lt{2^m}(x)$ is irreducible over $\mathbb{Q}$ if and only if $d(x)^2-g(x)\rho^2$ is irreducible over $\mathbb{Q}\left(\rho^2\right)$ for some $\rho\in R$.
\end{theorem}

\begin{proof}
For all $z\in\mathbb{C}$ such that $g(z)\neq 0$, we can deduce that
\begin{equation*}
\alpha\Lt{2^m}(z)=g(z)^{2^{m-1}}L_{2^m}\left(\frac{d(z)}{g(z)^{1/2}}\right).
\end{equation*}
From Lemma \ref{genHogattLucasLemma} we know that  $L_{2^m}(x)\in\mathbb{Z}[x^2]$. Let $G_{2^m}(x)\in\mathbb{Z}[x]$  
such that $G_{2^m}(x^2)=L_{2^m}(x)$. Since $L_{2^m}(x)$ is irreducible over $\mathbb{Q}$, it follows that $G_{2^m}(x)$ is irreducible over $\mathbb{Q}$.  
For all $z\in\mathbb{C}$ such that $g(z)\neq 0$, we deduce
\begin{equation*}
\alpha\Lt{2^m}(z)=g(z)^{2^{m-1}}G_{2^m}\left(\frac{d(z)^2}{g(z)}\right). 
\end{equation*}
 The conclusion follows from Lemma \ref{JCCapelliLemma}.
\end{proof}

Computer experimentation shows that there are many other polynomials $d(x)$ and $g(x)$ such  $\Lt{q}(x)/p_1(x)$ and  $\Ft{q}(x)$ are irreducible for primes greater than $2$ and $\Lt{2^k}(x)$ is irreducible over $\mathbb{Q}$. 

 \section{Acknowledgement}

The first author was partially supported by the Citadel Foundation. The second author was supported by a postdoctoral fellowship at the University of Calgary.

\noindent  MSC 2010:
Primary 11B39; Secondary 11B83.

\end{document}